\documentclass[twoside,leqno,twocolumn]{article}

\usepackage[letterpaper]{geometry}

\usepackage{ltexpprt}
\usepackage{hyperref}

\usepackage{amsmath} 
\usepackage{amssymb}  
\usepackage{graphicx}
\usepackage{mathtools}
\usepackage{thmtools}
\usepackage{lipsum}
\usepackage{amsfonts}
\usepackage{graphicx}
\usepackage{epstopdf}
\usepackage[ruled,vlined]{algorithm2e}
\usepackage{color}

\newcommand{\R}{\mathbb{R}}
\def\conv{\mathrm{conv}}
\DeclareMathOperator*{\argmin}{argmin}
\def\Bern{\mathrm{Bern}}

\newtheorem{assumption}{Assumption}

\newtheorem{assumption2}{Assumption}

\begin{document}

\newcommand\relatedversion{}
\renewcommand\relatedversion{\thanks{This article benefited from the support of the FMJH Program PGMO and from the support to this program from EDF. N.\@ Oudjane was partially supported by the FiME Lab Research Initiative (Institut Europlace de Finance).}} 

\title{\Large Decomposed resolution of finite-state aggregative optimal control problems\relatedversion}
\author{Kang Liu\thanks{CMAP, Ecole Polytechnique, CNRS, Institut Polytechnique de Paris, Palaiseau, France and Inria, Laboratoire des Signaux et des Systèmes, CentraleSupélec, CNRS, Université Paris-Saclay, Gif-sur-Yvette, France,
 {\tt\small kang.liu@polytechnique.edu}}
\and Nadia Oudjane\thanks{OSIRIS department, EDF Lab, Paris-Saclay, France,
 {\tt\small nadia.oudjane@edf.fr}}
\and Laurent Pfeiffer\thanks{Inria, Laboratoire des Signaux et des Systèmes, CentraleSupélec, CNRS, Université Paris-Saclay, Gif-sur-Yvette, France,
 {\tt\small laurent.pfeiffer@inria.fr}}}

\date{}

\maketitle







\begin{abstract} \small\baselineskip=9pt A class of finite-state and discrete-time optimal control problems is introduced. The problems involve a large number of agents with independent dynamics, which interact through an aggregative term in the cost function. The problems are intractable by dynamic programming. We describe and analyze a decomposition method that only necessitates to solve at each iteration small-scale and independent optimal control problems associated with each single agent. When the number of agents is large, the convergence of the method to a nearly optimal solution is ensured, despite the absence of convexity of the problem.
The procedure is based on a method called Stochastic Frank-Wolfe algorithm, designed for general nonconvex aggregative optimization problems. Numerical results are presented, for a toy model of the charging management of a battery fleet.
\end{abstract}

\section{Introduction.}
This article is dedicated to a class of aggregative optimal control problems in discrete time and discrete state space. These problems involve a large number $N$ of agents, indexed by $i=1,\ldots,N$, and time steps ranging over $t=0,1,\ldots,T$.
For any agent $i$, we fix a finite \emph{state set} $S_i$ and a finite \emph{control set} $U_i$. The evolution of agent $i$ is described by \emph{transition functions} $\pi_i^t \colon S_i \times U_i \rightarrow {S}_i$, where $t=0,\ldots,T$. We also fix mappings $U_i^t \colon S_i \rightarrow 2^{U_i}$ describing the feasible controls of the agents: at time $t$, if the agent $i$ is in state $s_i^t$, he can make use of all controls in $U_i^t(s_i^t)$.
The initial state of each agent $i$ is constrained to be in $S^0_i$, a subset of $S_i$.
The problem also involves some functions $f_t\colon \R \rightarrow \R$ which we call \emph{social} cost (at time $t$) and some functions $h_i^t \colon S_i^t \times U_i \rightarrow \R$, which we call \emph{contribution} functions. We also make use of functions $\ell_i^t  \colon S_i^t \times U_i \rightarrow \R$, which we call \emph{individual} costs.
The optimal control problem of interest reads:
\begin{equation}\label{pb:OC}
    \begin{cases}
    \begin{array}{rl}
    \underset{(s,u)}{\inf} & J(s,u) \coloneqq  \sum_{t=0}^{T} f_t \Big(
\frac{1}{N} \sum_{i=1}^N h_i^t(s_i^t,u_i^t)
\Big) \\[0.5em]
  & \qquad \qquad \quad  +\frac{1}{N}
\sum_{i=1}^N 
\sum_{t=0}^{T} \ell_i^t(s_i^t,u_i^t),
 \\[0.7em]
 \text{s.t.}
&
s_i^{t+1}= \pi_i^t(s_i^t,u_i^t), \; 
 u_i^t \in  U_i^t(s_i^t), \; s_i^0 \in S_i^0,
 \\
 & \forall t =0,1,\ldots, T-1, \; i =1,2,\ldots, N,
 \end{array}
    \end{cases}
\end{equation}
where $(s,u)= (s_i^t,u_i^t)_{i=1,\ldots,N}^{t=0,\ldots,T}$. The problem is motivated by energy management problems involving flexibilities: they are small consumption (or production) units, which are able to shift their production or their consumption over time, typically by storing energy. We refer the reader to \cite{hao2014aggregate,Seguret2020}.
In these models, the dynamical systems are usually continuous in time and space, they can however be discretized in problems in the form \eqref{pb:OC}.
We focus here on fully discrete problems for simplicity, but also in order to emphasize the fact that our approach does not require any assumption on the dynamics of the agents: they could result from a discretization of a non-linear system, involving non-smooth terms. We will only require that the functions $f_t$ are convex, with a Lipschitz-continuous gradient.

A standard approach to deal with problem \eqref{pb:OC} relies on the dynamic programming principle (see \cite{bertsekas2012dynamic}), in which a key step is to compute the value function $V\colon \{0,1,\ldots, T\}\times S \rightarrow \mathbb{R}$, where the state space $S$ is defined by $\prod_{i=1}^N S_i$. For our problem, Bellman's equation reads as follows: for any $t \in \{ 0, \ldots T \}$, for any $s \in S$,
\begin{align*}
    V^t(s) = \min_{u \in U^t(s)} \, &  f_t \Big(
\frac{1}{N} \sum_{i=1}^N h_i^t(s_i,u_i) \Big)  + \frac{1}{N}\sum_{i=1}^N \ell^t(s_i,u_i)\\
 &  + V^{t+1}(\pi^t(s,u)),
\end{align*}
where $U^t(s) = \prod_{i=1}^N U^t_i(s_i)$ and where $\pi^t(s,u) = (\pi_i^t(s_i,u_i))_{i=1}^N$. We observe that the complexity of Bellman's equation increases exponentially with $N$; this phenomenon is the well-known curse of dimensionality. 
As a consequence, the dynamic programming approach is not tractable for problem \eqref{pb:OC} when the number of agents $N$ is large. Let us mention here that problem \eqref{pb:OC} can be formulated as a large-scale mixed integer convex program (MICP), that is to say, an optimization problem with integrity constraints which becomes a convex program if the integrity constraints are removed. We refer to \cite{coey2020outer} and to the references therein for the resolution of such problems.
However, the number of variables in the MICP corresponding to our problem is equal to $\sum_{i=1}^N T|S_i||U_i|$ and is quickly prohibitive for solvers such as GUROBI \cite{gurobi2018gurobi} and SCIP \cite{scip2021} as $N$ increases.

Another general approach for solving large-scale optimal control problems (in particular) and large-scale optimization problems with an aggregative structure (in general) relies on decomposition. Decomposition methods are particularly powerful when the associated Lagrangian has a separable structure, in which case the dual criterion takes the form of a sum that can be evaluated in parallel, making the resolution of the dual problem easier, with dual subgradient methods \cite[Chapter 6]{bertsekas1999}, with cutting plane algorithms \cite[Chapter XII, Section 4]{jbhu93} or with the alternating direction of multipliers method (ADMM) \cite{bertsekas2014constrained}.
Such approaches have been successfully applied to stochastic optimal control problem, see for example \cite{barty2010,carpentier2018stochastic,Seguret2020}.
However, the convergence of such methods is only guaranteed in the case of convex optimization problems, in general, which limits their application to linear dynamical systems.

This article provides a full description of a decomposition method that allows to circumvent the curse of dimensionality.
Our method is a stochastic method that can find, given $\epsilon > 0$, a $\mathcal{O}(1/N)$-optimal solution with probability $1-\epsilon$. 
At each iteration, only small-scale optimal control problems associated with each agents need to be solved.
Our method relies on the Stochastic Frank-Wolfe (SFW) algorithm proposed by the authors in the recent work \cite{blopw22}, dedicated to abstract aggregative optimization problems, as defined in \cite{Wang2017}.
The SFW algorithm is a combination of (i) the classical Frank-Wolfe algorithm, which is applied to a relaxed version of the optimization problem, with (ii) a selection method that allows to recover a solution to the original problem. The SFW relies on an oracle, which must solve efficiently some subproblems associated with each agent. It turns out that in the present context, those subproblems are small-scale optimal control problems which can be solved by dynamic programming.
A key point in the analysis of the SFW algorithm is the fact that the relaxed problem is a good approximation of the original problem. We have obtained an estimate for the relaxation gap in \cite{blopw22} of order $\mathcal{O}(1/N)$, improving previous results by Wang in \cite{Wang2017}, when $N$ is large. Let us note that another numerical method for aggregative problems is proposed in that reference, relying on Shapley-Folkman decompositions \cite{Starr1969}. We refer the reader to  \cite[Section 5.1]{blopw22} for a thorough comparison of the two methods.

This article is organized as follows. Sections \ref{sec:abstract} to \ref{sec:sfw} provide the reader with a summary of the SFW algorithm and some of the theoretical findings of the article \cite{blopw22}. In Section \ref{sec:abstract}, we give a general formulation of aggregative optimization problems. Their convex relaxation is introduced in Section \ref{sec:convex}. We decribe the SFW algorithm and its convergence properties in Section \ref{sec:sfw}.
Section \ref{sec:aoc} provides a reformulation of the optimal control problem \eqref{pb:OC} as a general aggregative problem of the form \eqref{pb:aggre} and details the implementation of the SFW algorithm for \eqref{pb:OC}.
We briefly discuss the MICP approach in Section \ref{sec:micp}.
Numerical simulations on the charging management of a battery fleet are presented in Section \ref{sec:app}.

\section{Abstract aggregative problems}
\label{sec:abstract}

The general aggregative optimization problem investigated in \cite{Wang2017} and \cite{blopw22} is given by
\begin{align*} \label{pb:aggre} \tag{P}
& \inf_{x\in \mathcal{X}}J(x) := f(G(x)), \\[0.2em]
& \text{where: } 
\mathcal{X} = \prod_{i=1}^N \mathcal{X}_i
\quad
\text{and}
\quad
G(x)= \frac{1}{N} \sum_{i=1}^{N}g_i(x_i).
\end{align*}
The sets $\mathcal{X}_i$ are given and the maps $g_i$ are defined from $\mathcal{X}_i$ to some Hilbert space $\mathcal{E}$. The function $f$ is defined from $\mathcal{E}$ to $\R$. An interpretation of problem \eqref{pb:aggre} is as follows: $N$ is the number of agents; the agents are indexed by $i$ and each variable $x_i \in \mathcal{X}_i$ corresponds to the decision attributed to agent $i$. The mapping $g_i$ is the contribution of agent $i$ to some common good, defined by $\frac{1}{N} \sum_{i=1}^N g_i(x_i)$. We will refer to it as the aggregate. The function $f$ is the social cost associated with the aggregate. In addition to applications in aggregative optimal control mentioned in the introduction, problem \eqref{pb:aggre} itself has important applications in various domains, such as the \emph{resource allocation problems} \cite{beaude2020privacy,jacquot2018analysis}, supervised learning problems, see \cite{Chizat2018,mei2018mean,mei2019mean}.

We make some non-restrictive structural assumptions on \eqref{pb:aggre}.
The aggregate space $\mathcal{E}$ is assumed to be the Cartesian product of $M$ Hilbert spaces $\mathcal{E}_1,\ldots,\mathcal{E}_M$. Moreover,  we suppose that $f$ is of the form
\begin{equation*}
f(y)
=
\sum_{j=1}^M
f_j(y_j), \quad
\forall y=(y_1,\ldots,y_M) \in \mathcal{E}.
\end{equation*}
The functions $f_j$ are defined from $\mathcal{E}_j$ to $\R$, for $j=1,\ldots,M$. Finally, the contribution mappings $g_i$ are of the form $g_i(x_i)= (g_{ij}(x_i))_{j=1,\ldots,M}$.
Therefore, the cost functional $J$ reads:
\begin{equation*}
J(x)
=
\sum_{j=1}^M
f_j \Bigg(
\frac{1}{N}
\sum_{i=1}^N
g_{ij}(x_i)
\Bigg).
\end{equation*}
Some general notations will be used all along the article: given two subsets $A$ and $B\subseteq\mathcal{E}$, we denote by $A+B$ their Minkowski sum. Given $\lambda \in \R$, we denote $\lambda A$ the set defined by $\{ \lambda a \mid a \in A \}$. We denote by $\conv(A)$ the convex hull of $A$.
Next we introduce the main assumptions of the work.
For any $i=1,\ldots,N$ and for any $j=1,\ldots,M$, we denote
\begin{equation*}
Y_{ij} = \big\{ g_{ij}(x_i) \mid x_i\in\mathcal{X}_i \big\}
\quad
\text{and}
\quad
Y_{j} = \frac{1}{N} \sum_{i=1}^N Y_{ij}.
\end{equation*}

\begin{assumption}\label{ass1}  
 For $i=1,2,\ldots,N$ and $j=1,2\ldots,M$:
\begin{itemize}
\item The range set $Y_{ij}$ has finite diameter $d_{ij}$.
\item The function $f_j $ is $L_j$-Lipschitz on $\conv(Y_j)$.
\item The function $f_j$ is continuously differentiable on a neighborhood of $\conv(Y_{j})$, and $\nabla f_j$ is $\tilde{L}_j-$Lipschitz on $\conv(Y_j)$.
\end{itemize}	
\end{assumption}

We next define two constants $C_0$ and $C_1$ by
\begin{align*}
C_0 = \ & \sum_{j=1}^{M} \Big( L_j
\max_{1 \leq i \leq N} \left\{ d_{ij} \right\} \Big), \\
C_1= \ & \frac{1}{N} \sum_{j=1}^M \Bigg( \tilde{L}_j \sum_{i=1}^N d_{ij}^2 \Bigg).
\end{align*}

\begin{assumption}\label{ass2}
For all $j= 1,\ldots,M$, the function $f_j \colon \mathcal{E}_j\rightarrow \mathbb{R}$ is convex.
\end{assumption}

\begin{assumption} \label{ass3}
For all $i=1,\ldots,N$ and for all $y \in \conv(G(\mathcal{X}))$, the problem
\begin{equation} \label{eq:sub_pb_i}
\inf_{x_i \in \mathcal{X}_i} \, \langle \nabla f(y), g_i(x_i) \rangle
\end{equation}
has at least one solution.
For all $i=1,\ldots,N$, we fix a map $\mathbb{S}_i \colon \conv(G(\mathcal{X})) \mapsto \mathcal{X}_i$ such that, for any $y \in \conv(G(\mathcal{X}))$, $\mathbb{S}_i(y)$ is a solution to \eqref{eq:sub_pb_i}.
\end{assumption}

\section{Convex relaxation}
\label{sec:convex}

We introduce in this section a convex relaxation of problem \eqref{pb:aggre}. It will motivate the stochastic Frank-Wolfe algorithm presented in the following section. The reader only interested in a practical implementation of the algorithm can move to the next section.
We first need to reformulate problem \eqref{pb:aggre}.
Let us define
\begin{equation*}
\mathcal{Y}_i = g_i(\mathcal{X}_i), \quad \forall i=1,\ldots,N,
\quad \text{and}
\quad
\mathcal{Y} = \frac{1}{N} \sum_{i=1}^N
\mathcal{Y}_i.
\end{equation*}
Problem \eqref{pb:aggre} is equivalent to
\begin{equation*}
\inf_{y \in \mathcal{E}}
f(y),
\quad
\text{subject to: }
y \in \mathcal{Y}.
\end{equation*}
Indeed, by definition of $\mathcal{Y}$, any $y \in \mathcal{E}$ lies in $\mathcal{Y}$ if and only if there exists $x \in \mathcal{X}$ such that
$y= \frac{1}{N} \sum_{i=1}^N g_i(x_i).$
For such an $x$, we have $f(y)=J(x)$.
It is natural to consider the following relaxation:
\begin{equation*} \label{pb:relaxed} \tag{PR}
\inf_{y \in \mathcal{E}}
f(y), \quad
\text{subject to: }
y \in \conv(\mathcal{Y}).
\end{equation*}
Under Assumption \ref{ass2}, the relaxed problem is a convex optimization problem.
Let $J^*$ denote the value of problem \eqref{pb:aggre} and let $\mathcal{J}^*$ denote the value of problem \eqref{pb:relaxed}.
We have the following result.

\begin{proposition} \label{prop:gap}
Let Assumption \ref{ass1} hold true. Then
\begin{equation*}
\mathcal{J}^*
\leq
J^{*} \leq \mathcal{J}^{*} + \frac{C_1}{2N}.
\end{equation*}
\end{proposition}

\begin{proof}
The first inequality is straightforward, and the second one is proved in \cite[Proposition 2.6]{blopw22}.
\end{proof}

Let us mention that a more precise upper bound is given in \cite[Theorem 4.4]{blopw22}. The result of Proposition \ref{prop:gap} is to be related to Shapley-Folkman's lemma \cite{Starr1969} and more precisely to Starr's corollary \cite{Starr1969}, which gives a bound of the distance to $\mathcal{Y}$ of a point in $\conv(\mathcal{Y})$.
Assuming that the coefficients $d_{ij}$ (appearing in Assumption \ref{ass1}) are uniformly bounded, we see that $C_1$ is also bounded, and thus the gap estimate $\frac{C_1}{2N}$ goes to zero as $N$ goes to infinity. In words, there is a convexification of the problem as the number of agents increases.

\begin{Remark} \label{rem:mfg}
Consider the particular case where the sets $\mathcal{X}_i$ and the contribution functions $g_i$ do not depend on $i$. Then, $\mathcal{Y}_1= \ldots = \mathcal{Y}_N$. It follows that
\begin{align*}
\conv(\mathcal{Y})= \ &
\conv \Bigg( \frac{1}{N} \sum_{i=1}^N \mathcal{Y}_1 \Bigg)
= \frac{1}{N} \sum_{i=1}^N \conv(\mathcal{Y}_1) \\
= \ & \conv(\mathcal{Y}_1).
\end{align*}
In this case, the relaxed problem does not depend on $N$, it can be interpreted as a mean-field relaxation, i.e., it can be interpreted as the limit problem as the number of agents $N \to \infty$.
\end{Remark}

From Proposition~\ref{prop:gap}, we note that any $\epsilon$-solution of problem~\eqref{pb:relaxed}  is an $\epsilon$-solution of problem~\eqref{pb:aggre} as soon as it is feasible for problem~\eqref{pb:aggre}.
Since problem \eqref{pb:relaxed} is convex, it is easier to handle numerically. Algorithm \ref{alg:fw_geo} generates a minimizing sequence $(y^k)_{y \in \mathbb{N}}$ in $\conv(\mathcal{Y})$ for the relaxed problem, by a direct application of the Frank-Wolfe algorithm \cite{Dunn1978}. 
The general idea of Algorithm \ref{alg1+k} is to introduce an approximation step at each iteration, to recover points in $\mathcal{Y}$.

\begin{algorithm}[htb]
\SetAlgoLined
Initialization: $y^0 \in \conv(\mathcal{Y})$\;
\For{$k= 0,1,2,\ldots$}{
Find a solution $\bar{y}^k$ to the sub-problem
\begin{equation} \label{eq:sub_pb_algo}
\inf_{y \in \conv(\mathcal{Y})}
\langle \nabla f(y^k), y \rangle
\end{equation}

Choose $\omega_k \in [0,1]$\;
Set $y^{k+1} = (1- \omega_k) y^k + \omega_k \bar{y}^k$\;
}
\caption{Frank-Wolfe Algorithm for the relaxed problem}
\label{alg:fw_geo}
\end{algorithm}

The algorithm is known to converge for various choices of the parameter $\omega_k$. In particular, for $\omega_k= 2/(k+2)$, one can show the existence of a constant $C>0$ such that for any $k$, $f(y^k)
\leq \mathcal{J}^*
+ \frac{C}{k}$.
Besides the guaranty of convergence of the algorithm, its interest lies in the decomposability of the sub-problems to be solved at each iteration. Problem \eqref{eq:sub_pb_algo} is indeed equivalent to
\begin{equation} \label{eq:sub-problem}
\inf_{x \in \mathcal{X}}
\Big\langle \nabla f(y^k), \sum_{i=1}^N g_i(x_i)
\Big\rangle.
\end{equation}
Obviously, $x$ is a solution if and only if $x_i$ is a solution to \eqref{eq:sub_pb_i} (with $y=y^k$). Therefore a solution to \eqref{eq:sub_pb_algo} is given by
\begin{equation}
\bar{y}^k
=
\frac{1}{N} \sum_{i=1}^N g_i \big( \mathbb{S}_i(y^k) \big).
\end{equation}
Note that $\bar{y}^k \in \mathcal{Y}$. However, even if $y^k$ also belonged to $\mathcal{Y}$, there would be no reason to have $y^{k+1} \in \mathcal{Y}$.

The stochastic Frank-Wolfe algorithm (Algorithm \ref{alg1+k}) introduced in the next section allows us to overcome this difficulty. At the iteration $k$, a point $x^k$ has been constructed, with aggregate $y^k= \frac{1}{N} \sum_{i=1}^N g_i(x_i^k)$. The same sub-problems are solved, yielding a point $\bar{x}^k= (\mathbb{S}_1(y^k),\ldots,\mathbb{S}_N(y^k))$ with aggregate $\bar{y}^k$. Next, the algorithm generates a sample of $n_k$ points independently and identically distributed (i.i.d.) in $\mathcal{X}$, denoted $\hat{x}^{k,j}= (\hat{x}_i^{k,j})_{i=1,\ldots,N}$, with $j=1,\ldots,n_k$.
The point $\hat{x}_i^{k,j}$ 
is equal to $x_i^k$ with probability $1-\omega_k$ and to $\bar{x}_i^{k,j}$ with probability $\omega_k$. 
In practice, we simulate $Nn_k$ i.i.d. random variables $\lambda_i^{k,j}\sim \Bern(\omega_k)$, where $\Bern(\omega_k)$ denotes the Bernoulli distribution of parameter $\omega_k \in [0,1]$ and we set
\begin{equation*}
\hat{x}_i^{k,j}
= (1-\lambda_i^{k,j}) x_i^k
+ \lambda_i^{k,j} \bar{x}_i^k.
\end{equation*}
Then $x^{k+1}$ is taken as a minimizer of $J$ over the union 
of the set of points randomly generated and $\{ x^k \}$.
\section{Stochastic Frank-Wolfe algorithm}
\label{sec:sfw}

We provide in Algorithm \ref{alg1+k} an explicit implementation of our stochastic Frank-Wolfe algorithm. 

\begin{algorithm}[htb]
\SetAlgoLined
Initialization: $x^0 \in \mathcal{X}$\;
\For{$k= 0,1,2,\ldots$}{
\medskip
\textbf{Step 1: Resolution of the sub-problems.}\\
Compute $y^k= \frac{1}{N} \sum_{i=1}^N g_i(x_i^k)$\;
\For{$i=1,2,\ldots,N$}{
Compute $\bar{x}_i^k = \mathbb{S}_i (y^k) $\;
}
\medskip
\textbf{Step 2: Update.} \\
Choose $n_k \in \mathbb{N}^*$ and $\omega_k \in [0,1]$\;
\For{$j=1,2,\ldots,n_k$}{
\For{$i=1,2,\ldots,N$}{
Simulate $\lambda^{k,j}_i \sim \Bern(\omega_k)$, independently of all previously defined random variables\;
Set $\hat{x}_i^{k,j} = (1-\lambda_i^{k,j})x_i^k + \lambda_i^{k,j} \bar{x}_i^k$\;
}
Set $\hat{x}^{k,j}= (\hat{x}^{k,j}_i)_{i=1,\ldots,N}$\;
}
Find $x^{k+1} \in \argmin \{ J(x) \, \big| \, x \in X^k \}$,
where
$X^k= \{ \hat{x}^{k,j},\, j=1,2,\ldots,n_k \} \cup \{ x^k \}$\;
\medskip
}
\caption{Stochastic Frank-Wolfe Algorithm }
\label{alg1+k}
\end{algorithm}

We have the following result, proved in  \cite[Theorem 3.7]{blopw22}.

\begin{theorem} \label{thm:prob}
Let Assumptions \ref{ass1}, \ref{ass2}, and \ref{ass3} hold true.
Assume that $\omega_k= \frac{2}{k+2}$, for all $k \in \mathbb{N}$ in Algorithm \ref{alg1+k}. Then, for all $K =1, \ldots, 2N$, 
\begin{equation*}
\mathbb{E} [\gamma_K] \leq \frac{4C_1}{K}, \quad \text{where $\gamma_K= J(x^K)- \mathcal{J}^*$}. 
\end{equation*}
Moreover, for all $\epsilon > 0$,
\begin{equation*}
\mathbb{P} \Big[
\gamma_K < \frac{4C_1}{K} + \epsilon
\Big] \geq 1 - \exp \left( \frac{-\varepsilon^2 N}{2(v_K + \epsilon m_K/3)} \right),
\end{equation*}
where the constants $m_K$ and $v_K$ are given by
\begin{align*}
v_K= \ & \frac{2 C_0^2}{K^2 (K+1)^2} \ \Bigg(
{\displaystyle \sum_{k=1}^{K-1}}
\frac{k(k+1)^2}{n_k} \Bigg) \\
m_K= \ & \frac{C_0}{K(K+1)} \
\Big( \,
{\displaystyle \max_{k=1,\ldots,K-1}
}
\frac{(k+1)(k+2)}{n_k}
\Big).
\end{align*}
\end{theorem}

Note that the constants $v_K$ and $m_K$ can be made arbitrarily small by choosing sufficiently large values of $(n_k)_{k=0,\ldots,K}$. Thus for arbitrarily small values of $\epsilon > 0$ and $\epsilon' > 0$, one can choose appropriate numbers of random simulations so that
$\mathbb{P} \big[
\gamma_{2N} < \frac{2C_1}{N} + \epsilon
\big] \geq 1 - \epsilon'$.

\begin{Remark}
Theorem \ref{thm:prob} focuses on the choice of stepsize $\omega_k= 2/(k+2)$, which we have utilized in the numerical simulations. It is also possible to determine $\omega_k$ by line search, see \cite[Remark 3.10]{blopw22}.
\end{Remark}

\section{Aggregative optimal control}
\label{sec:aoc}

In this section, we reformulate optimal control problem \eqref{pb:OC} as a problem of the form \eqref{pb:aggre}. We address its resolution with the SFW algorithm.

\subsection{Abstract formulation}

Let us consider an agent $i$ and let us describe its state-control feasible set $\mathcal{X}_i$. Recall that a state $S_i$, a control set $U_i$, mappings $U_i^t \colon S_i \rightarrow 2^{U_i}$, and transition mapping $\pi_i^t \colon S_i \times U_i \rightarrow S_i$ are given. We call feasible \emph{state-control trajectory} an element $x_i = (s_i,u_i)$, where $s_i = (s_i^0,\ldots,s_i^T) \in ({S}_i)^{T+1}$ and $u_i= (u_i^0,\ldots,u_i^{T}) \in ({U}_i)^{T+1}$, such that
\begin{equation*}
s_i^0 \in S_i^0, \quad
u_i^t \in U_i^t(s_i^t), \quad
s_i^{\theta+1}= \pi_i^\theta(s_i^\theta,u_i^\theta),
\end{equation*}
for any $t= 0, \ldots, T$ and any $\theta= 0,\ldots,T-1$.
We denote by $\mathcal{X}_i$ the set of feasible state-control trajectories. We set $\mathcal{X}= \prod_{i=1}^{N^{\phantom{h}}} \mathcal{X}_i$.
The non-emptyness of $\mathcal{X}_i$ is a straightforward consequence of the following assumption.

\begin{assumption2} \label{ass:non-empty}
The set $S_i^0$ is non-empty. For all $t=0,\ldots,T$, for all $s_i^t \in S_i$, the set $U_i^t(s_i^t)$ is non-empty.
\end{assumption2}

With each agent $i$ are associated $(T+1)$ contribution functions $h_i^t$, $t=0,\ldots,T$ and $(T+1)$ individual costs $\ell_i^t$, $t=0,\ldots,T$. We set $\mathcal{E}_0=\ldots \mathcal{E}_{T+1}=\R$ and we define $T+2$ functions $g_{it} \colon \mathcal{X}_i \rightarrow \mathcal{E}_t$ by
\begin{equation*}
g_{it}(x_i)
= \begin{cases}
\begin{array}{ll}
h_{i}^t(s_i^t,u_i^t) & \text{ if $t \leq T$} \\
\sum_{t'=0}^T \ell_i^{t'}(s_i^{t'},u_i^{t'}) & \text{ if $t=T+1$}.
\end{array}
\end{cases}
\end{equation*}
The social costs $f_0,\ldots,f_T$ are the same as in the original problem \eqref{pb:OC}. The social cost $f_{T+1} \colon \mathcal{E}_{T+1} \rightarrow \R$ is the identity function.
With these definitions, problem \eqref{pb:OC} is equivalent to 
\begin{align} \label{eq:oc}
\inf_{(x_i)_{i=1}^N \in \prod_{i=1}^N \mathcal{X}_i}  \ \ \sum_{t=0}^{T+1}
f_t \Big( \frac{1}{N} \sum_{i=1}^N g_{it}(x_i) \Big).
\end{align}

\subsection{Assumptions}

As before, we denote $g_i(x_i)= (g_{it}(x_i))_{t=0,\ldots,T+1}$, $\mathcal{E}= \prod_{t=0}^{T+1} \mathcal{E}_t = \R^{T+2}$ and for $y \in \mathcal{E}$, $f(y)= \sum_{t=0}^{T+1} f_t(y_t)$.
For any $i=1,\ldots,N$ and for any $t=0,\ldots,T+2$, we denote
\begin{equation*}
Y_{it} = \big\{ g_{it}(x_i) \mid x_i \in \mathcal{X}_i \big\}
\quad \text{and} \quad
Y_{t} = \frac{1}{N} \sum_{i=1}^N Y_{it}.
\end{equation*}

\begin{assumption2}\label{ass:all_assumptions}  
For $i=1,2,\ldots,N$ and for $t=0,1,\ldots,T$,
\begin{itemize}
\item $f_t $ is $L_t$-Lipschitz on $\conv(Y_t)$,
\item $f_t$ is continuously differentiable on a neighborhood of $\conv(Y_t)$, $\nabla f_t$ is $\tilde{L}_t$-Lipschitz on $\conv(Y_t)$
\item $f_t$ is convex on $\conv(Y_t)$.
\end{itemize}
\end{assumption2}

Assumptions \ref{ass:non-empty} and \ref{ass:all_assumptions} imply Assumptions \ref{ass1} and \ref{ass2} for problem \eqref{eq:oc}. Assumption \ref{ass3} is trivially satisfied since $\mathcal{X}_i$ is a finite set.

\subsection{Resolution of the sub-problems}

We explain now how to solve the sub-problems \eqref{eq:sub_pb_i} associated with the aggregative optimal control problem \eqref{eq:oc}.
Let $y \in \mathcal{E}$. Let $\mu \in \mathcal{E}$ be defined by
$\mu^t= \nabla f_t (y^t)$.
By definition of $f_{T+1}$, $\mu^{T+1}=1$.
The sub-problem \eqref{eq:sub_pb_i} reads:
\begin{align} \label{eq:sub_pb_oc}
\inf_{x_i \in \mathcal{X}_i} \
 \sum_{t=0}^{T} \Big(
\ell_i^t(s_i^t,u_i^t) +
\langle \mu^t, h_i^t(s_i^t,u_i^t) \rangle \Big).
\end{align}
The sub-problem \eqref{eq:sub_pb_oc} can be solved by dynamic programming. Algorithm \ref{algo:dyn_prog} yields a solution to \eqref{eq:sub_pb_oc}. For convenience, we denote
\begin{equation*}
\ell_i^t[\mu^t](s_i^t,u_i^t)
= \ell_i^t(s_i^t,u_i^t) +
\langle \mu^t, h_i^t(s_i^t,u_i^t) \rangle
\end{equation*}
in the algorithm.
The algorithm consists of two steps: first in a backward pass, a sequence of value functions $(V_i^t)_{t=0,\ldots,T+1}$ is computed, where $V_i^t \colon S_i \rightarrow \R$.
A globally optimal solution is obtained in a forward pass. Note that the value of the optimization problem of Step 1 is finite as a consequence of Assumption \ref{ass:non-empty}.
\begin{algorithm}[htb]
\SetAlgoLined

\medskip

\textbf{Step 1: Backward pass.} \\
Set $V_i^{T+1}(s_i^{T+1})=0$, for any $s_i^{T+1} \in S_i$. \\
\For{$t=T,T-1,\ldots,0$}{
\For{$s_i^t \in S_i$}{
Define $V_i^t(s_i^t)$ as
\begin{align*} 
\hspace{-3mm}
\begin{array}{rl}
{\displaystyle \min_{u_i^t \in U_i^t(s_i^t)}} & \! \! \! \!
\ell_i^t[\mu^t](s_i^t,\cdot)
+ V_i^{t+1} \big(\pi_i^t(s_i^t,\cdot) \big).
\end{array}
\end{align*}
}
}

\medskip

\textbf{Step 2: Forward pass.} \\
Find
$\bar{s}_i^0 \in
\underset{s_i^0 \in S_i^0}{\text{argmin}} \,
V_i^0(s_i^0)$\;
\For{$t=0,\ldots,T$}{
Find a solution $\bar{u}_i^t$ to the problem
\begin{align}
\hspace{-3mm}
\begin{array}{rl}
{\displaystyle \min_{U_i^t(\bar{s}_i^t,u_i^t)}} & \! \! \! \!
\ell_i^t[\mu^t](\bar{s}_i^t,\cdot)
+ V_i^{t+1} \big(\pi_i^t(\bar{s}_i^t,\cdot) \big).
\end{array}
\notag
\end{align}
If $t<T$, set $\bar{s}_i^{t+1}= \pi_i^t(\bar{s}_i^t, \bar{u}_i^t)$.
}

\medskip

\caption{Dynamic programming algorithm}
\label{algo:dyn_prog}
\end{algorithm}

\begin{Remark}
As mentioned in the introduction, problem \eqref{eq:oc} could be addressed by dynamic programming. This would allow the computation of an exact solution. However, this would require to compute a value function of the form $V^t(s^t)$, where $s^t= (s_1^t,\ldots,s_N^t) \in \prod_{i=1}^N S_i$. The resulting complexity, of order $T \prod_{i=1}^N |S_i|$, is prohibitive even for moderate values of $N$. In contrast, the complexity of each iteration of the stochastic Frank-Wolfe algorithm is linear with respect to $N$, while the accuracy of the algorithm improves as $N$ increases.
\end{Remark}


\section{MICP formulation} \label{sec:micp}

We give another equivalent problem of \eqref{pb:OC} in this section as a mixed integer convex program (MICP).
For any $t \in \{ 0,\ldots, T\}$, we denote $Z_i^t =\{(s_i^t,u_i^t) \in S_i \times U_i\, \mid \, u_i^t \in U_i^t(s) \}$. The optimization variable are denoted $m= (m_i^t(s_i^t,u_i^t))$ with indices $i \in \{ 1,\dots, N \}$, $t \in \{ 0, \ldots, T \}$, $(s_i^t,u_i^t) \in Z_i^t$. The criterion is defined as:
\begin{align*}
\bar{J}(m) \coloneqq &
\sum_{t=0}^T
f_t \Big(
\frac{1}{N} \sum_{i=1}^N
\sum_{z_i^t \in Z_i^t} h_i^t(z_i^t) m_i^t(z_i^t) \Big) \\
& + \frac{1}{N} \sum_{t=0}^T \sum_{i=1}^N \sum_{z_i^t \in Z_i^t} \ell_i^t(z_i^t) m_i^t(z_i^t),
\end{align*}
where for simplicity we have written $z_i^t$ instead of $(s_i^t,u_i^t)$.
We call now (MICP) the problem which consists in minimizing $\bar{J}(m)$ over the variables $m$ satisfying the following constraints:
\begin{align*}
\text{(i)} \ \ & m_i^t(s_i^t,u_i^t) \in \mathbb{Z}, \\
\text{(ii)} \ \ & m_i^t(s_i^t,u_i^t) \geq 0, \quad \sum_{z_i^t \in Z_i^t} m_i^t(z_i^t) = 1, \\
\text{(iii)} \ \ & m_i^0(\hat{s}_i^0,u_i^0)= 0, \\
\text{(iv)} \ \ & \sum_{u_i^\theta \in U_i^\theta(s_i^\theta)} m_i^\theta(s_i^\theta,u_i^\theta)
= \sum_{z_i^{\theta-1} \in (\pi_i^\theta)^{-1}(s_i^{\theta})} m_i^{\theta-1}(z_i^{\theta-1}),
\end{align*}
for any $i=1,\ldots,N$, $t=0,\ldots,T$, $(s_i^t,u_i^t) \in Z_i^t$, $\hat{s}_i^0 \in S_i \backslash S_i^0$, $\theta = 1, \ldots, T$, and $s_i^\theta \in S_i$.
The terminology MICP comes from the fact that if the integrity constraint in the problem, constraint (i), then the problem becomes a nonlinear convex program: indeed, $\bar{J}$ is a convex function and constraints (ii)-(iii)-(iv) are affine.
Let us emphasize the fact that in problem \eqref{pb:OC}, the states $s_i^t$ and the controls $u_i^t$ are directly optimized, while here, we optimize a variable $m$ indexed by all possible states and controls.
The following lemma shows the equivalence between \eqref{pb:OC} and (MICP).

\begin{lemma}
Problems \eqref{pb:OC} and (MICP) have the same value; moreover, from any solution to \eqref{pb:OC}, a solution to (MICP) can be deduced and vice versa.
\end{lemma}

\begin{proof}
Let $(\bar{s}, \bar{u})$ be a solution of \eqref{pb:OC}. Take
\begin{equation*}
    m_t^i(s_i^t,u_i^t) = \begin{cases}
    1, \qquad &\text{if } (s_i^t,u_i^t) = (\bar{s}_i^t, \bar{u}_i^t),
    \\
    0, & \text{otherwise}.
    \end{cases}
\end{equation*}
We can verify that $m$ is feasible for (MICP) and that $J(\bar{s},\bar{u}) = \bar{J}(m)$.

Let $\bar{m}$ be a solution of (MICP). From the constraints (i-ii), we deduce that for any $(t,i)$, there exists a unique $(s_i^t,u_i^t) \in Z_t^i$, such that $\bar{m}(t,i,s_i^t,u_i^t) = 1$. Constraint (iii) implies that $s_i^0 \in S_i^0$. Finally constraint (iv) implies that $s_i^{t+1} = \pi_i^t(s^t_i,u^t_i)$. We also have $J(s,u) = \bar{J}({\bar{m}})$. The conclusion follows.
\end{proof}


\section{Application example}
\label{sec:app}

Let us now turn to the problem of the charging of a fleet of batteries. We propose a very simple model which is essentially illustrative, rather than realistic. 
However, it is emphasised that the proposed approach can easily incorporate more realistic constraints on battery operation (e.g. taking into account limits on cycles numbers). Indeed, these refinements remain localized at the sub-problem level (impacting only the dynamic programming Algorithm~\ref{algo:dyn_prog}). They consist either in adding a state variable or in modifying the local costs in order to penalise undesired behaviour.
Suppose that there are $N$ batteries to be charged. Let $s_i^t$ be the state of charge (SoC) for the battery $i$ at the time $t$.

\subsection{Dynamics of the batteries}

The dynamics of each battery is characterized by three parameters: an initial state of charge ${s}^{\text{in}}_i \in \mathbb{N}$, a maximal state of charge $s_i^\text{{max}} \in \mathbb{N}$, a maximal load speed $u_i^{\text{max}} \in \mathbb{N}$.
We define:
\begin{align*}
& S_i= \{ s_i^{\text{in}},\ldots,s_i^{\text{max}} \},
\ S_i^0 = \{ s_i^{\text{in}}\}, \ U_i= \{ 0,\dots,u_i^{\text{max}} \}, \\
& U_i^t(s_i^t)= \{ 0,\ldots, \min(u_i^{\max}, s_i^{\text{max}}- s_i^t) \}, \\
& \pi_i^t(s_i^t,u_i^t)= s_i^t + u_i^t.
\end{align*}
In words: the initial condition $s_i^{\text{in}}$ is given, the charging of the battery is additive, the charging speed is bounded by $u_i^{\text{max}}$ and is such that $s_i^t$ can never exceed $s_i^{\max}$.

\subsection{Cost and contribution functions}

Some positive coefficients $(\beta_i)_{i=1,\ldots,N}$, $(\alpha_t)_{t=0,\ldots,T-1}$, and $(c_t)_{t=0,\ldots,T-1}$ are given. The individual costs are
\begin{align*}
\ell_i^t(s_i^t,u_i^t)= \ & 0, \quad \forall t=0,\ldots,T-1, \\
\ell_i^T(s_i^T,u_i^T)= \ & \beta_i (s_i^{\text{max}}- s_i^T)^2.
\end{align*}
The contributions are defined by $h_i^T(s_i^T,u_i^T)= 0$ and
\begin{align*}
h_i^t(s_i^t,u_i^t)= u_i^t, \quad \forall t=0,\ldots,T-1.
\end{align*}
The social costs $f_t$ are defined by $f_T(y_T)=0$ and
\begin{align*}
f^t(y_t)= \alpha^t(y_t - c_t)^2, \quad \forall t=0,\ldots,T-1.
\end{align*}
Therefore, the cost function $J$ reads
\begin{align*}
\sum_{t=0}^{T-1}
\alpha^t \Bigg( \Big( \frac{1}{N} \sum_{i=1}^N u_i^t \Big) - c^t \Bigg)^2 + \frac{1}{N} \sum_{i=1}^N \beta_i \left( s_i^T - s_i^{\text{max}} \right)^2.
\end{align*}
The cost function has two contributions, one depends on the average of charging levels of all the batteries, the other one depends on the individual final SoC of each battery. To be more precise, for $t\leq T-1$, the average charging level needs to approach some target power $c_t$. For $t= T$, the batteries expect to approach their maximum SoCs.

\subsection{Numerical simulations}

The parameters are chosen as follows:
\begin{itemize}
\item $N=100$, $T=24 $
\item $s_i^{\text{in}}$ (resp.\@ $s_i^{\text{max}}$) is chosen randomly and uniformly in $\{ 0,1,\ldots, 20 \}$ (resp.\@ $\{ 20,21,\ldots, 40 \}$), $u_i^{\text{max}}= 4$
\item $\alpha^t$ is chosen randomly and uniformly in $[1,2]$, $\beta_i$ is chosen randomly and uniformly in $[0,1] $
\item $c^t= 1.5 \lfloor \sin (\pi t /12)+1 \rfloor$.
\end{itemize}
Thus, for $t = 0,1 ,\ldots,23$,  the diameter of the range set $Y_{it}$ is less than $u_i^{\text{max}} = 4$, and the Lipschitz constant $ \tilde{L}_t$ is $2 \alpha^t $, which is less than $4$. Then, we have the following upper bound for the relaxation gap $C_1 / 2N$:
\begin{align*}
    \frac{C_1}{2N}  \leq \frac{1}{200} \cdot \frac{1}{100} \cdot \sum_{t=0}^{23} \left(4 \cdot \sum_{i=1}^{100} 4^2 \right) = 7,68 . 
\end{align*}

Let us estimate the number of variables in the MICP corresponding to this example, that we denote by $d(m)$. First,
$|Z_{i}^t|\geq (s_i^{\text{max}}-s_i^{\text{in}} - u_i^{\text{max}})  (u_i^{\text{max}} +1 )$.
Then, in view of the distributions of $s_i^{\text{max}}$ and $s_i^{\text{in}}$,
\begin{equation*}
\mathbb{E}(d(m)) \geq N T \mathbb{E}(|Z_{i}^t|) \geq 100 \cdot 24 \cdot 16 \cdot 5 = 192\, 000,
\end{equation*}
making the MICP approach intractable.

Fig.\@ \ref{fig1} shows the outcome of Algorithm \ref{alg:fw_geo} with $500$ iterations to get an approximation of the minimum $\mathcal{J}^*$ of the relaxed problem. The curve represents the relaxed cost. Fig.\@ \ref{fig2} shows the outcome of Algorithm \ref{alg1+k}, for different  choices of $n_k$ with 100 iterations. Since the algorithm is stochastic, 
we ran it 50 times independently to evaluate its efficiency; 
the curves represent the average value of $\gamma_k= J(x^k)-\mathcal{J}^*$. The standard deviation is displayed on Fig.\@ \ref{fig3}. In all cases, an average value of the gap significantly smaller than $7,68$ can be reached; the standard deviation is also significantly smaller than $7,68$ at the last iterations. We have initialized the algorithm with values of $x_i^0$ such that $u_i^t= 0$, for any $t=0,\ldots,T-1$.


\begin{figure}[htbp]
\centerline{\includegraphics[width= 3.3 in]{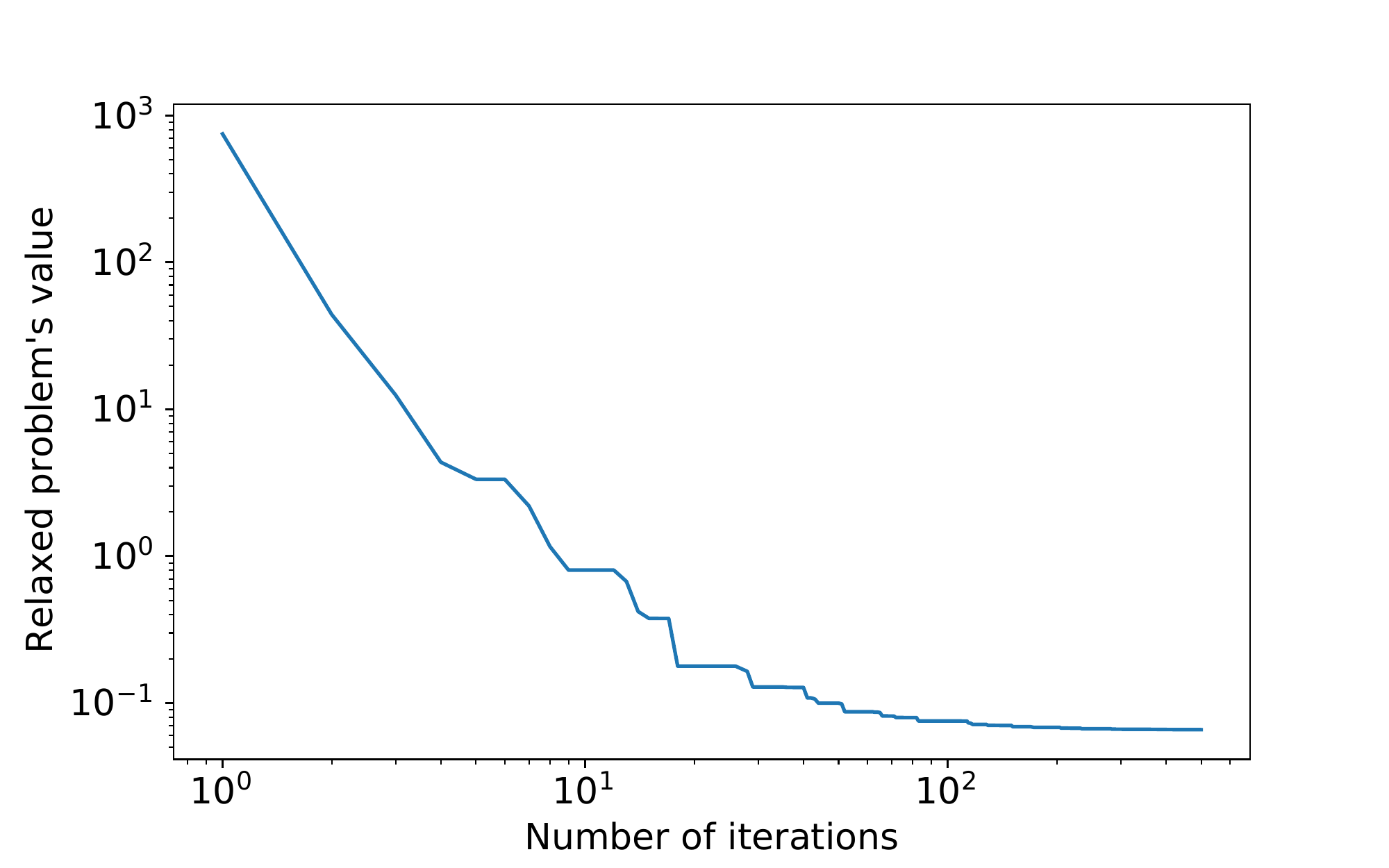}}
\caption{Frank-Wolfe Algorithm with $500$ iterations for the relaxed problem.}
\label{fig1}
\end{figure}

\begin{figure}[htbp]
\centerline{\includegraphics[width= 3.3 in]{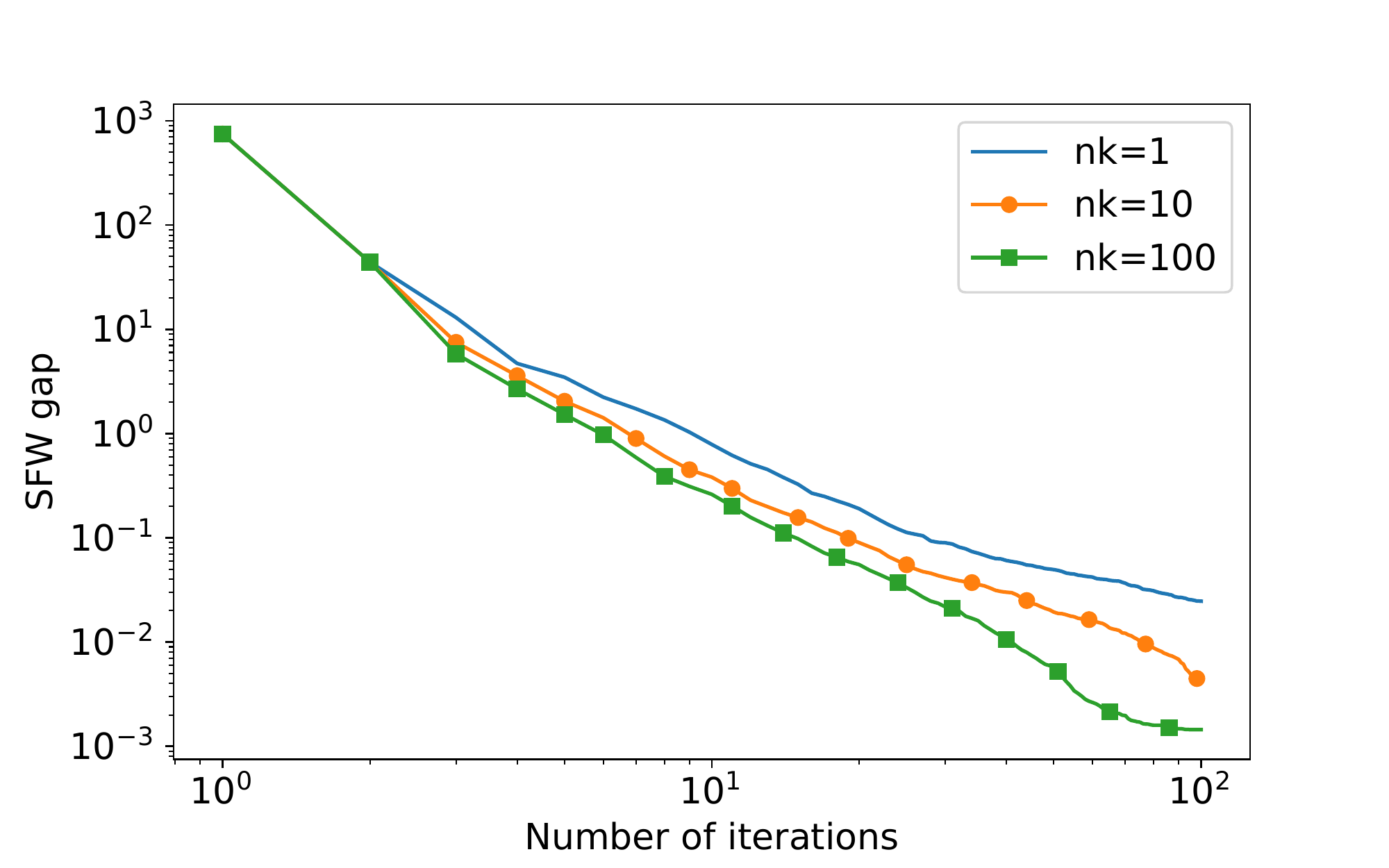}}
\caption{Stochastic Frank-Wolfe Algorithm with $100$ iterations, expectation of the gap.}
\label{fig2}
\end{figure}

\begin{figure}[htbp]
\centerline{\includegraphics[width= 3.3 in]{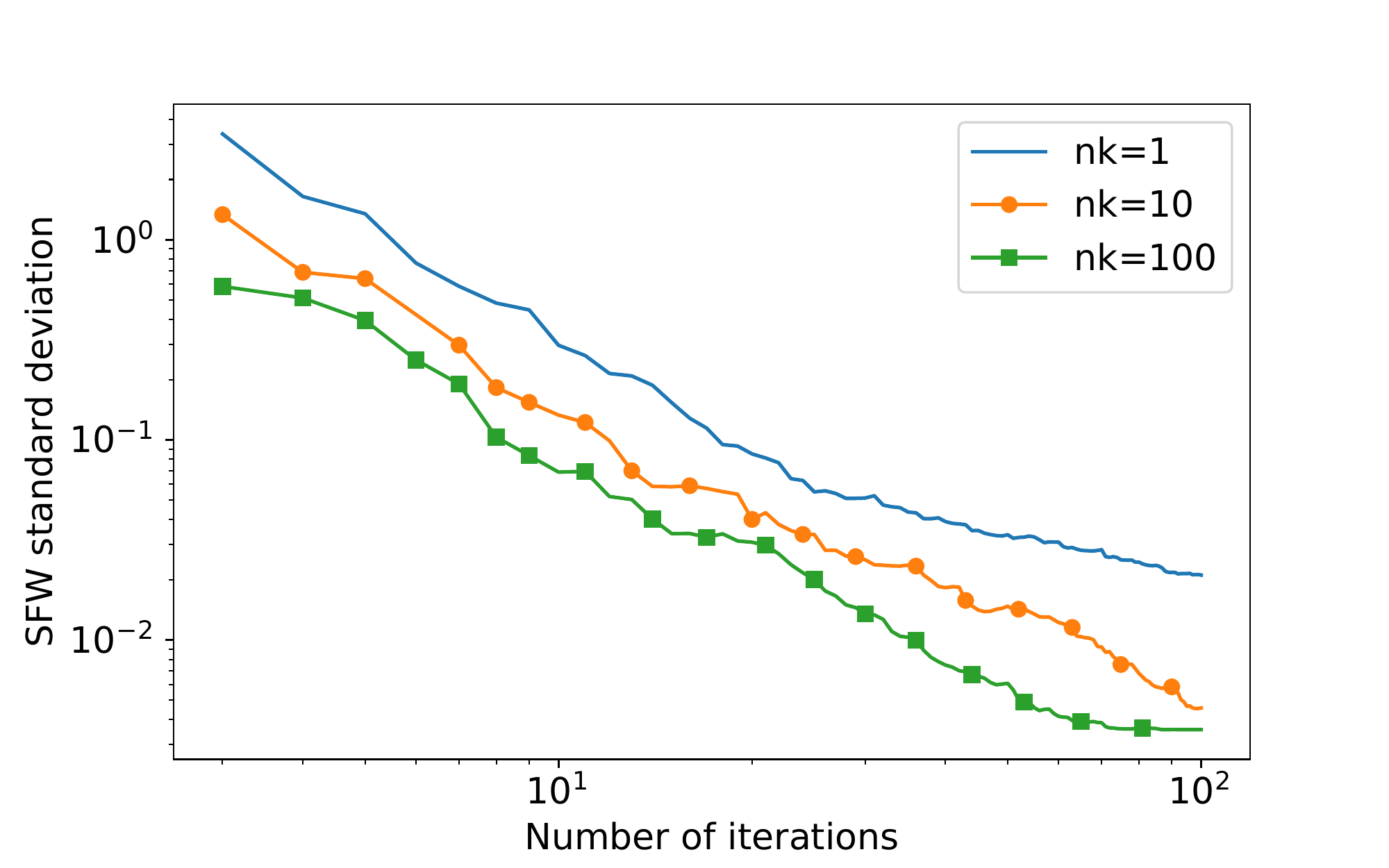}}
\caption{Stochastic Frank-Wolfe Algorithm with $100$ iterations, standard deviation of the gap.}
\label{fig3}
\end{figure}

\bibliographystyle{plain}
\bibliography{biblio}

\begin{thebibliography}{10}

\bibitem{barty2010}
K.~Barty, P.~Carpentier, and P.~Girardeau.
\newblock Decomposition of large-scale stochastic optimal control problems.
\newblock {\em RAIRO-Operations Research}, 44(3):167--183, 2010.

\bibitem{beaude2020privacy}
O.~Beaude, P.~Benchimol, S.~Gaubert, P.~Jacquot, and N.~Oudjane.
\newblock A privacy-preserving method to optimize distributed resource
  allocation.
\newblock {\em SIAM Journal on Optimization}, 30(3):2303--2336, 2020.

\bibitem{bertsekas2012dynamic}
D.~Bertsekas.
\newblock {\em Dynamic programming and optimal control: Volume I}, volume~1.
\newblock Athena scientific, 2012.

\bibitem{bertsekas1999}
D.P. Bertsekas.
\newblock {\em Nonlinear programming.}
\newblock Belmont, MA: Athena Scientific, 2nd ed. edition, 1999.

\bibitem{bertsekas2014constrained}
D.P. Bertsekas.
\newblock {\em Constrained optimization and Lagrange multiplier methods}.
\newblock Academic press, 2014.

\bibitem{scip2021}
K.~Bestuzheva, M.~Besan{\c{c}}on, W.K. Chen, A.~Chmiela, T.~Donkiewicz, J.~van
  Doornmalen, L.~Eifler, O.~Gaul, G.~Gamrath, A.~Gleixner, et~al.
\newblock The scip optimization suite 8.0.
\newblock {\em arXiv preprint arXiv:2112.08872}, 2021.

\bibitem{blopw22}
J.F. Bonnans, K.~Liu, N.~Oudjane, L.~Pfeiffer, and C.~Wan.
\newblock Large-scale nonconvex optimization: randomization, gap estimation,
  and numerical resolution.
\newblock {\em ArXiv preprint}, 2022.

\bibitem{carpentier2018stochastic}
P.~Carpentier, J.-P. Chancelier, V.~Lecl{\`e}re, and F.~Pacaud.
\newblock Stochastic decomposition applied to large-scale hydro valleys
  management.
\newblock {\em European Journal of Operational Research}, 270(3):1086--1098,
  2018.

\bibitem{Chizat2018}
L.~Chizat and F.~Bach.
\newblock On the global convergence of gradient descent for over-parameterized
  models using optimal transport.
\newblock In {\em Advances in Neural Information Processing Systems},
  volume~31. Curran Associates, Inc., 2018.

\bibitem{coey2020outer}
C.~Coey, M.~Lubin, and J.P. Vielma.
\newblock Outer approximation with conic certificates for mixed-integer convex
  problems.
\newblock {\em Mathematical Programming Computation}, 12(2):249--293, 2020.

\bibitem{Dunn1978}
J.C. Dunn and S.~Harshbarger.
\newblock Conditional gradient algorithms with open loop step size rules.
\newblock {\em Journal of Mathematical Analysis and Applications},
  62(2):432--444, 1978.

\bibitem{gurobi2018gurobi}
LLC Gurobi~Optimization.
\newblock Gurobi optimizer reference manual, 2018.

\bibitem{hao2014aggregate}
H.~Hao, B.M. Sanandaji, K.~Poolla, and T.L. Vincent.
\newblock Aggregate flexibility of thermostatically controlled loads.
\newblock {\em IEEE Transactions on Power Systems}, 30(1):189--198, 2014.

\bibitem{jbhu93}
J.-B. Hiriart-Urruty and C.~Lemar{\'e}chal.
\newblock {\em Convex analysis and minimization algorithms {II}: {Advanced}
  theory and bundle methods}, volume 306 of {\em Grundlehren Math. Wiss.}
\newblock Berlin: Springer-Verlag, 1993.

\bibitem{jacquot2018analysis}
P.~Jacquot, O.~Beaude, S.~Gaubert, and N.~Oudjane.
\newblock Analysis and implementation of an hourly billing mechanism for demand
  response management.
\newblock {\em IEEE Transactions on Smart Grid}, 10(4):4265--4278, 2018.

\bibitem{mei2019mean}
S.~Mei, T.~Misiakiewicz, and A.~Montanari.
\newblock Mean-field theory of two-layers neural networks: dimension-free
  bounds and kernel limit.
\newblock In {\em Conference on Learning Theory}, pages 2388--2464. PMLR, 2019.

\bibitem{mei2018mean}
S.~Mei, A.~Montanari, and P.-M. Nguyen.
\newblock A mean field view of the landscape of two-layer neural networks.
\newblock {\em Proceedings of the National Academy of Sciences},
  115(33):E7665--E7671, 2018.

\bibitem{Seguret2020}
A.~Seguret, C.~Alasseur, J.F. Bonnans, A.~De~Paola, N.~Oudjane, and V.~Trovato.
\newblock Decomposition of high dimensional aggregative stochastic control
  problems.
\newblock {\em ArXiv preprint}, 2021.

\bibitem{Starr1969}
R.M. Starr.
\newblock Quasi-equilibria in markets with non-convex preferences.
\newblock {\em Econometrica: journal of the Econometric Society}, pages 25--38,
  1969.

\bibitem{Wang2017}
M.~Wang.
\newblock Vanishing price of decentralization in large coordinative nonconvex
  optimization.
\newblock {\em SIAM J.\@ Optim.\@}, 27(3):1977--2009, 2017.

\end{thebibliography}
\end{document}